\documentclass[11pt,a4paper]{article}
\usepackage{amsmath,enumerate,bm,bbm,color,fullpage}
\usepackage{amsthm}
\usepackage{amssymb}
\usepackage{amsbsy}
\usepackage{amsfonts}
\usepackage{amstext}
\usepackage{amscd}

\numberwithin{equation}{section}
\theoremstyle{plain}

\newtheorem{theorem}{Theorem}[section]
\newtheorem{corollary}[theorem]{Corollary}
\newtheorem{lemma}[theorem]{Lemma}
\newtheorem{proposition}[theorem]{Proposition}
\theoremstyle{definition}
\newtheorem{definition}[theorem]{Definition}
\newtheorem{remark}[theorem]{Remark}

\newtheorem{example}[theorem]{Example}
\newtheorem{conjecture}[theorem]{Conjecture}
\newtheorem{problem}[theorem]{Problem}

\def\C{{\mathbb C}}
\def\N{{\mathbb N}}
\def\R{{\mathbb R}}

\def\CC{{\mathcal C}}

\def\d{{\rm d}}
\def\6{\, {\rm d}}
\def\ri{{\rm i}}
\def\e{{\rm e}}
\def\ep{{\varepsilon}}
\def\im{{\sf Im}}
\def\re{{\sf Re}}
\def\supp{{\sf supp}}
\def\cc{{\mathbf c}}
\def\fU{{\rm U}(\boxplus)}
\def\cU{{\rm U}(\ast)}
\def\fSD{{\rm SD}(\boxplus)}
\def\cSD{{\rm SD}(\ast)}

\def\AA{{\rm (At)}}
\def\DD{{\rm (De)}}
\def\UUi{{\rm (U1)}}
\def\UUii{{\rm (U2)}}

\def\MP{{\bm \pi}}
\def\fst{{\rm\bf f}}

\def\be{\begin{equation}}
\def\ee{\end{equation}}

\begin{document}

\title{Unimodality for free L\'evy processes}

\author{\sc Takahiro Hasebe \and \sc Noriyoshi Sakuma}

\maketitle

\begin{abstract}
We will prove that: (1) A symmetric free L\'evy process is unimodal if and only if its free L\'evy measure is unimodal; (2) Every free L\'evy process with boundedly supported L\'evy measure is unimodal in sufficiently large time. (2) is completely different property from classical L\'evy processes. 
On the other hand, we find a free L\'evy process such that its marginal distribution is not unimodal for any time $s>0$ and its free L\'evy measure does not have a bounded support. 
Therefore, we conclude that the boundedness of the support of free L\'evy measure in (2) cannot be dropped. 
For the proof we will (almost) characterize the existence of atoms and the existence of continuous probability densities of marginal distributions of a free L\'evy process in terms of L\'evy--Khintchine representation.  
\end{abstract}

keywords: free probability, free convolution, free L\'evy process, unimodality

Mathematics Subject Classification 2010: 46L54, 60G51

\section{Introduction} 
A Borel measure $\mu$ on $\R$ is said to be {\it unimodal} if, for some $c\in\R$,
\begin{equation}\label{UMa}
\mu(\d t)=\mu(\{c\})\delta_c(\d t)+f(t)\6t,
\end{equation}
where  $f\colon\R\to[0,\infty)$ is non-decreasing on $(-\infty,c)$ and non-increasing on
$(c,\infty)$. In this case $c$ is called the {\it mode}. A stochastic process is said to be unimodal if the marginal distributions are all unimodal. 

Unimodality in the context of free probability was investigated first by Biane \cite{BP99} who proved that all free stable laws are unimodal, and then by Haagerup and Thorbj{\o}rnsen \cite{HT14} who proved that free gamma distributions are unimodal and by Hasebe and Thorbj{\o}rnsen \cite{HT} who proved that all freely selfdecomposable distributions are unimodal, generalizing the past results. We continue research on unimodality in free probability. We have two main results  in this paper. 
\begin{itemize}  
\item[\UUi] A symmetric free L\'evy process is unimodal if and only if its free L\'evy measure is unimodal. 
\item[\UUii] Every free L\'evy process with boundedly supported L\'evy measure is unimodal in sufficiently large time.  
\end{itemize} 

$\UUi$ and $\UUii$ will be proved in Theorem \ref{free Jurek unimodal} and in Theorem \ref{unimodal-large} respectively. We will investigate other properties on the marginal distributions of free L\'evy processes:  
\begin{itemize}
\item[\AA] Characterizing the existence of atoms in terms of free L\'evy--Khintchine representation; 
 
\item[\DD] Almost characterizing the continuity of the probability density functions in terms of free L\'evy--Khintchine representation.  
\end{itemize} 

These results will be used in the proofs of $\UUi$ and $\UUii$. 

The background of $\UUi$ and $\UUii$ traces back to Yamazato's theorem in 1978 proving that all classical selfdecomposable distributions are unimodal \cite{Y78}. After Yamazato's theorem there have been contributions to the study of unimodality by Sato, Watanabe, Yamazato and others, see \cite{Wat01}. However, a necessary and sufficient condition for an infinitely divisible (ID) distribution or a L\'evy process to be unimodal is not known in terms of the L\'evy measure and Gaussian component. 
Characterizing  unimodal ID probability measures seems a difficult question, but the characterization of unimodal {\it symmetric} L\'evy processes is known in terms of the unimodality of the L\'evy measure.  

\begin{theorem}[Medgyessy \cite{M67}, Wolfe \cite{W78}]\label{sym unimodal}
Let $\mu$ be symmetric and ID. 
The following statements are equivalent. 
\begin{enumerate}[\rm(1)] \setlength{\itemindent}{3mm}
\item $\mu^{\ast s}$ is unimodal for any $s>0$. 
\item The L\'evy measure of $\mu$ is unimodal (with mode $0$). 
\end{enumerate}
Note that if $\mu$ is symmetric ID then its L\'evy measure is also symmetric.  Hence the mode of the L\'evy measure must be $0$ if it is unimodal. 
\end{theorem}
Medgyessy showed the implication (2)$\Rightarrow$(1) and Wolfe showed the converse. 
When $\mu$  is not symmetric, the implication (1)$\Rightarrow$(2) still holds true as shown by Wolfe \cite{W78}, but (2) does not imply (1). Actually  Wolfe gave the following example. 

\begin{proposition}[Wolfe \cite{W78}]\label{Wolfe ex} Let $\mu$ be an ID  distribution without a Gaussian component. Suppose that its L\'evy measure is a probability measure with mean $m\neq0$ and variance $\sigma^2<\infty$. Then $\mu^{\ast s}$ is not unimodal for $s>\frac{3\sigma^2}{m^2}$. 
\end{proposition}
  Hasebe and Thorbj{\o}rnsen \cite{HT} proved the free version of Yamazato's theorem: All freely selfdecomposable distributions are unimodal. In the present paper we will prove $\UUi$, i.e.\ the free analog of Theorem \ref{sym unimodal}, thus finding another similarity between classical and free L\'evy processes in addition to Yamazato's theorem. Wolfe's Proposition \ref{Wolfe ex} says that for a class of L\'evy processes, the unimodality fails to hold in large time. In free probability, the opposite conclusion holds; we can show $\UUii$ saying that all free L\'evy processes with boundedly supported L\'evy measure are unimodal in large time. Thus a sharp difference on unimodality appears between classical and free L\'evy processes as well as similarities.

The background of $\AA$ is also some classical result: The existence of atoms in a classical convolution semigroup $(\mu^{\ast s})_{s \geq0}$ can be characterized in terms of the L\'evy--Khintchine representation. Recall that a measure $\mu$ on $\R$ is said to be continuous if $\mu(\{x\})=0$ for any $x\in\R$. 
\begin{theorem}[See \cite{Sa99}, Theorem 27.4]\label{thm continuity}
If $\mu$ is ID, then the following are equivalent: 
\begin{enumerate}[\rm(1)]\setlength{\itemindent}{3mm}
\item\label{c condition 1} $\mu^{\ast s}$ is not continuous for some $s>0$;  
\item\label{c condition 1.5} $\mu^{\ast s}$ is not continuous for any $s>0$;  
\item\label{c condition 2} $\mu$ is of type A. 
\end{enumerate} 

\end{theorem}
We will study atoms  and try to show the free analog of Theorem \ref{thm continuity}, but the complete analog fails; the free analog of assertion \eqref{c condition 1} does not imply the free analog of \eqref{c condition 1.5} since a free convolution semigroup does not have an atom in large time \cite[Proposition 5.12]{BV93} (note that the statement in \cite{BV93} is only for discrete time $n\in\N$ but the proof applies to real time). However, we will show that the free analogs of assertions \eqref{c condition 1} and \eqref{c condition 2} are equivalent.   

We will prove the existence of a continuous density on $\R$ under some assumptions, which seems to have no classical counterpart. In particular, any free convolution semigroup in large time becomes absolutely continuous with respect to the Lebesgue measure with a {\it continuous} probability density function. 

The proofs of our results on $\AA$ and $\DD$ are based on Huang's necessary and sufficient condition for the existence of an atom and Huang's density formula \cite{Hu2}, respectively.  The proofs of the main results $\UUi$ and $\UUii$ are based on Huang's density formula,  $\AA$, $\DD$ and the methods developed in \cite{HT14} and \cite{HT}. 

The remaining sections are organized as follows. Section \ref{preliminaries} contains basic knowledge on free probability required in this paper. We will review classical and free ID distributions and then Huang's density formula for FID distributions. 
Section \ref{Atom} contains results on atoms and the continuity of probability density functions.  
Section \ref{Unimodal Jurek} contains the rigorous statement of $\UUi$ and its proof. We will include several examples of probability measures in the free Jurek class and also in the class of freely selfdecomposable distributions. 
Section \ref{Unimodal Large} contains the rigorous statement of $\UUii$ and its proof. Then we find an {\it unbounded} free L\'evy process whose marginal distribution is not unimodal at any time, thus showing that we cannot remove the assumption of boundedness in $\UUii$. Throughout the paper several open questions are presented.


\section{Preliminaries} \label{preliminaries}

\subsection{ID distributions}
We collect some concepts and results on ID distributions that appeared in Introduction and that will motivate definitions in Section \ref{sec FID}. We refer the reader to \cite{GK68,Sa99,SvH04} for details.  A probability measure on $\R$ is said to be {\it ID} ({\it infinitely divisible}) if it has an $n^{\rm th}$ convolution power root for any $n\in\N$ (this $n^{\rm th}$ root is actually unique).  A probability measure $\mu$ is ID if and only if its characteristic function has the {\it L\'evy--Khintchine representation} 
\begin{equation}\label{classical LK}
\hat{\mu}(u)=\exp\Big[{\rm i}\eta_\mu u - {\textstyle\frac{1}{2}}a_\mu u^2 + 
\int_{{\mathbb R}}\big({\rm e}^{{\rm i}ut}-1-{\rm i}ut 1_{[-1,1]}(t)\big)\nu_\mu({\rm d}t)\Big], \qquad u\in\R, 
\end{equation}
where $\eta_\mu$ is real, $a_\mu \geq 0$ (called the {\it Gaussian component}) and $\nu_\mu$ (called the {\it L\'evy measure}) is a nonnegative measure on $\R$ satisfying 
\be 
\nu_\mu(\{0\})=0,\qquad\int_\R \min\{1,t^2\}\,\nu_\mu(\d t) <\infty. 
\ee 
The triplet $(\eta_\mu,a_\mu,\nu_\mu)$ is called the {\it characteristic triplet}.

\begin{definition} Let $\mu$ be an ID distribution and let $\nu$ be its L\'evy measure. 
\begin{enumerate}[\rm(1)]
\item The measure $\mu$ is said to be {\it s-selfdecomposable} if $\nu$ is unimodal with mode 0. 
The set of s-selfdecomposable distributions is denoted by $\cU$. The class $\cU$ is called the {\it Jurek class} (see \cite{J85}). 
\item The measure $\mu$ is said to be {\it selfdecomposable} if the measure $|t| \nu(\d t)$ is unimodal with mode 0. 
The set of selfdecomposable distributions is denoted by $\cSD$. 
\end{enumerate}
\end{definition}
By definition we have the inclusion $\cSD \subset \cU$.  

In Theorem \ref{thm continuity} the following terminology was used (see \cite{Sa99}). 
\begin{definition}
 An ID distribution $\mu$ on $\R$ is of {\it type A} if its characteristic triplet $(\eta_\mu,a_\mu, \nu_\mu)$ satisfies $a_\mu=0$ and $\nu_\mu(\R)<\infty$. 
\end{definition}
An ID distribution $\mu$ is of type A if and only if $\mu=\delta_c \ast \rho$ for some $c\in \R$ and a compound Poisson distribution $\rho$. 

\subsection{FID distributions}\label{sec FID}
Let $G_\mu$ be the \emph{Cauchy transform} of a probability measure $\mu$ on $\R$
\be
G_\mu(z) := \int_{\R}\frac{1}{z-x}\,\mu(dx), \qquad z\in \C^+, 
\ee
and let $F_\mu$ be the reciprocal of $G_\mu$, that is 
\be
F_\mu(z):=\frac{1}{G_\mu(z)}, \qquad z\in \C^+,
\ee 
called the \emph{reciprocal Cauchy transform} of $\mu$. We define the truncated cone 
\be
\Gamma_{\lambda, M}:=\{z \in \C^+\mid \im(z) >M,~ |\re(z)| <\lambda \im(z)\}.
\ee 
In \cite{BV93} it was proved that for any $ \lambda>0$, there exists $\alpha, \beta, M>0$ such that $F_\mu$ is univalent in $\Gamma_{\alpha, \beta}$ such that $F_\mu(\Gamma_{\alpha, \beta}) \supset \Gamma_{\lambda, M}$, and so the right compositional inverse map $F_\mu^{-1}\colon\Gamma_{\lambda, M}\to \C^+$ exists such that $F_\mu \circ F_\mu^{-1}=\text{\normalfont Id}$ in $\Gamma_{\lambda, M}$. 

Then the {\it free cumulant transform} (or the $R$-transform) is defined by 
\be
\CC_\mu(z)=zF_\mu^{-1}(1/z)-1, \qquad 1/z \in \Gamma_{\lambda,M}.  
\ee 
This is a variant of the Voiculescu transform 
\be
\varphi_\mu(z):= F^{-1}_\mu(z)-z = z \CC_\mu(1/z), \qquad z\in \Gamma_{\lambda,M}. 
\ee
Then $\CC_\mu$ is the free analog of $\log\hat{\mu}$ since it linearizes free convolution:
\be
\CC_{\mu\boxplus\nu}(z)=\CC_\mu(z)+\CC_{\nu}(z)
\ee
for all $z$ in the intersection of the domains of the three transforms.

A probability measure on $\R$ is said to be {\it FID} ({\it freely infinitely divisible}) if it has an $n^{\rm th}$ convolution power root for any $n\in\N$.  
Bercovici and Voiculescu proved that $\mu$ is FID if and only if the Voiculescu transform $\varphi_\mu(z):=F^{-1}_\mu(z)-z$ has analytic continuation to a map from $\C^+$ taking values in $\C^-\cup\R$. This condition is equivalent to the condition that $-\varphi_\mu$ extends to a Pick function, and so it has the {\it Pick--Nevanlinna representation} (see \cite{BV93})
\be\label{fLK0}
\varphi_\mu(z) = -\gamma_\mu +\int_\R \frac{1+x z}{z-x}\,\sigma_\mu(\d x),\qquad z\in\C^+ 
\ee
for some $\gamma_\mu \in \R$ and a nonnegative finite measure $\sigma_\mu$ on $\R$. This representation can be rewritten in the form \cite{BNT02b}  
\be\label{free LK} 
\mathcal{C}_{\mu}(z) = \eta_\mu z+ a_\mu z^2 + 
\int_{{\mathbb R}}\Big(\frac{1}{1-t z}-1-t z1_{[-1,1]}(t)\Big)\nu_\mu(\d t), \qquad z\in \C^-, 
\ee
where $\eta_\mu\in\R, a_\mu \geq0$ and $\nu_\mu$ is a nonnegative measure on $\R$ satisfying 
\be 
\nu_\mu(\{0\})=0,\qquad\int_\R \min\{1,t^2\}\,\nu_\mu(\d t) <\infty. 
\ee
The formula \eqref{free LK} is called the {\it free L\'evy--Khintchine representation}. It has a correspondence with the classical L\'evy--Khintchine representation \eqref{classical LK}. 
The triplet $(\eta_\mu, a_\mu, \nu_\mu)$ is called the {\it free characteristic triplet}, $a_\mu$ is called the {\it semicircular component} and $\nu_\mu$ is called the {\it free L\'evy measure} of $\mu$. For an FID distribution $\mu$, the free convolution semigroup $\mu^{\boxplus s}, s \geq0,$ is defined to be the measure having the free characteristic triplet $(s \eta_\mu,s a_\mu,s \nu_\mu)$. 
Note that the finite measure $\sigma_\mu$ in \eqref{fLK0} and $\nu_\mu$ are related by the formula 
\be\label{fLM}
\nu_\mu(\d t) = \frac{1+t^2}{t^2}\sigma_\mu|_{\R\setminus\{0\}}(\d t).  
\ee

For a given ID distribution $\mu$ with characteristic triplet $(\eta_\mu,a_\mu,\nu_\mu)$, we can define an FID distribution $\Lambda(\mu)$ having the {\it free} characteristic triplet $(\eta_\mu,a_\mu,\nu_\mu)$. The bijection $\Lambda\colon{\rm ID} \to {\rm FID}$ is called the {\it Bercovici--Pata bijection} \cite{BP99}.  

We then define the free analog of the Jurek class that appeared in \cite{AH} and the class of selfdecomposable distributions introduced in \cite{B-NT02}.  

 \begin{definition}\label{def0} Let $\mu$ be FID and $\nu$ be its free L\'evy measure. 
 \begin{enumerate}[\rm(1)] 
\item The measure $\mu$ is said to be {\it freely s-selfdecomposable} if $\nu$ is unimodal with mode 0. 
The set of freely s-selfdecomposable distributions is denoted by  $\fU$ and is called the {\it free Jurek class}.
\item The measure $\mu$ is said to be {\it freely selfdecomposable} if the measure $|t| \nu(\d t)$ is unimodal with mode 0. The set of freely selfdecomposable distributions is denoted by  $\fSD$.
\end{enumerate}
\end{definition}
By definition, we have the inclusion $\fSD \subset \fU$, and in terms of the Bercovici--Pata bijection we have $\Lambda(\cSD) =\fSD$ and $\Lambda(\cU)=\fU$. 
A freely selfdecomposable distribution $\mu$ has a free L\'evy measure of the form $\nu_\mu (\d t) = \frac{k(t)}{|t|}\d t$ where $k$ is non-decreasing on $(-\infty,0)$ and non-increasing on $(0,\infty)$. Unless $\mu$ is a point measure or a semicircle distribution, $k \neq 0$ and so $k(0+)>0$ or $k(0-)>0$, and hence $\nu_\mu(\R)=\infty$. By contrast, there are freely $s$-selfdecomposable distributions $\mu$ whose free L\'evy measure satisfies $\nu_\mu(\R)<\infty$.

The probability distribution $\MP$ characterized by 
\be
\CC_{\MP}(z) = \frac{z}{1-z} 
\ee
is called the {\it standard free Poisson distribution}. It is known that for a probability measure $\sigma$ on $\R$ the free multiplicative convolution $\MP \boxtimes \sigma$ is the compound free Poisson distribution characterized by 
\be
\CC_{\MP\boxtimes \sigma}(z) = \int_{\R} \frac{t z}{1- t z}\sigma(\d t).  
\ee
This fact can be proved by using the $S$-transform as in \cite[Proposition 4]{P-AS12} when $\sigma$ is compactly supported with nonzero mean.  The general case is shown by approximation. Note that $\boxtimes$ is bi-continuous with respect to the uniform distance \cite{BV93}, but weak bi-continuity is still not known except the special case when both probability measures are supported on $[0,\infty)$.

\subsection{Atoms and probability density functions of FID distributions}
Let $\mu$ be an FID distribution. It is known that the singular continuous part of $\mu$ is zero \cite[Theorem 3.4]{BB04} and the number of atoms of $\mu$ is at most one \cite[Proposition 5.12]{BV93}, so 
\be
\mu = w \delta_c + \mu^{\rm ac}
\ee
for some $c\in\R$ and $w\in[0,1]$. 
 Moreover, Huang derived a formula for the absolutely continuous part $\mu^{\rm ac}$.  
Since $\mu$ is FID, the map $F^{-1}_\mu(z)=z + \varphi_\mu(z)$ extends to an analytic function in $\C^+$. Let  
\be
v_\mu(x):= \inf\{y>0\mid \im(F^{-1}_\mu(x+\ri y))>0 \},  
\ee
which is a continuous map on $\R$, and let  $\Omega:=F_\mu(\C^+)$. Then 
\be
\Omega = \{x+\ri y\mid x\in\R, y > v_\mu(x)\}. 
\ee
The map $F^{-1}_\mu$ extends to a homeomorphism from $\overline{\Omega}$ onto $\C^+\cup\R$ and then the map $x\mapsto x + \ri v_\mu(x)$ is a homeomorphism from $\R$ onto $\partial \Omega$. Thus one can define
\be
\psi_\mu(x):=F^{-1}_\mu(x+\ri v_\mu(x)),\qquad x\in\R, 
\ee
which is a homeomorphism of $\R$. For more details see \cite{Hu2} and also \cite{Hu1}.

\begin{theorem}[Huang \cite{Hu2}, Theorem 3.10] \label{thm Hu} Let $\mu$ be an FID distribution. 
Let $V_\mu=\{x\in\R\mid v_\mu(x)>0\}$. 
Then the support of the absolutely continuous part $\mu^{\rm ac}$ is $\psi(\overline{V_\mu})$ and  
\be\label{density Hu}
\frac{\d\mu^{\rm ac}}{\d x}(\psi_\mu(x)) = \frac{v_\mu(x)}{\pi(x^2+v_\mu(x)^2)}, \qquad x\in \R. 
\ee
Moreover, $\mu$ has an atom if and only if $v_\mu(0)=0$ and  
\be
\lim_{\ep\downarrow0} \frac{F_\mu^{-1}(\ri \ep) - F_\mu^{-1}(0)}{\ri \ep} = w >0, 
\ee
and in this case $\mu(\{F^{-1}_\mu(0)\})=w$. 
\end{theorem}
It is known that $\psi_\mu$ is real analytic in $V_\mu$ and hence so is $\d\mu^{\rm ac}/\d x$ in $\psi_\mu(V_\mu)$. This implies that if an FID distribution $\mu$ is unimodal then it is strictly unimodal, i.e.\ there is no plateau of the density. 

As an immediate consequence of Huang's formula, we prove an asymptotic property of the tail of an FID distribution. 
\begin{proposition}\label{tail} If $\mu$ is FID then 
\be
\lim_{|x|\to\infty}\frac{\d\mu^{\rm ac}}{\d x}(x)=0. 
\ee
\end{proposition}
\begin{proof}
If $v_\mu(x)>0$ then $\frac{x^2}{v_\mu(x)} + v_\mu(x) \geq 2 |x|$, and so by \eqref{density Hu} we have 
\be
\frac{\d\mu^{\rm ac}}{\d x}(\psi_\mu(x)) \leq \frac{1}{2 \pi|x|},\qquad x\neq 0.  
\ee
Since $\psi_\mu$ is a homeomorphism of $\R$ it satisfies $\lim_{|x|\to\infty} |\psi_\mu(x)|=\infty$, and the conclusion follows. 
\end{proof}

\section{Existence of atoms, continuity of density functions} \label{Atom}
We define the free analog of type A distributions via the Bercovici--Pata bijection. 
\begin{definition}
An FID distribution $\mu$ on $\R$ is of {\it free type A} if its free characteristic triplet $(\eta_\mu,a_\mu, \nu_\mu)$ satisfies $a_\mu=0$ and $\nu_\mu(\R)<\infty$. 
\end{definition} 
\begin{remark}
A probability measure $\mu$ is of free type A  if and only if $\mu=\delta_c \boxplus \rho$ for some $c\in \R$ and a compound free Poisson distribution $\rho$.  This is because the class of free type A distributions is the image of the type A distributions by the Bercovici--Pata bijection. The free L\'evy--Khintchine representation of a free type A distribution has the reduced form
\be\label{reducedfLK}
\CC_\mu(z) =  c_\mu z + \int_{\R}\left(\frac{1}{1-z t}-1\right) \nu_\mu(\d t), \qquad z \in \C^-, 
\ee
where $c_\mu \in \R$ and $\nu_\mu$ is the free L\'evy measure. 
\end{remark}

The main result of this section is: 
\begin{theorem}\label{atom existence}
If $\mu$ is FID, then the following are equivalent: 
\begin{enumerate}[\rm(1)]\setlength{\itemindent}{3mm}
\item\label{condition 1} $\mu^{\boxplus s}$ is not continuous for some $s>0$;  
\item\label{condition 2} $\mu$ is of free type A, 
\end{enumerate} 
and in this case $\mu^{\boxplus s}$ has an atom at $s F^{-1}_\mu(+\ri 0)$ with mass $1-s  \nu_\mu(\R)$ for $0\leq s < \nu_\mu(\R)^{-1}$, 
and $\mu^{\boxplus s}$ does not have an atom for  $ s \geq \nu_\mu(\R)^{-1}$. We understand that $\nu_\mu(\R)^{-1}=\infty$ if $\nu_\mu(\R)=0$, i.e.\ $\mu$ is a delta measure.  
\end{theorem}

This theorem follows from the following. 
\begin{theorem}\label{AC}
Let $\mu$ be FID and $(\eta_\mu,a_\mu,\nu_\mu)$ be its free characteristic triplet. 
\begin{enumerate}[\rm(1)]\setlength{\itemindent}{3mm}
\item\label{a1} If $a_\mu>0$ or $a_\mu=0$ and $\nu_\mu(\R)\in(1,\infty]$ then $\mu=\mu^{\rm ac}$ with continuous density function on $\R$.   
\item\label{a2} If $a_\mu=0$ and $\nu_\mu(\R)=1$ then $\mu=\mu^{\rm ac}$. 

\item \label{a3} If $a_\mu=0$ and $\nu_\mu(\R)\in[0,1)$ then the limit $F^{-1}_\mu(+\ri 0) \in\R$ exists and $\mu(\{F^{-1}_\mu(+\ri 0)\})=1-\nu_\mu(\R)$.  
\end{enumerate} 
\end{theorem}
\begin{remark}\label{rem0}
\begin{enumerate}[\rm(i)]
\item\label{rem0-1} Every selfdecomposable distribution satisfies $\nu_\mu(\R)=\infty$ unless it is a point measure or a semicircle distribution (see the paragraph following Definition \ref{def0}), so it is absolutely continuous with respect to the Lebesgue measure with continuous density on $\R$. This was also remarked in the end of \cite{HT}. Case \eqref{a3} shows that some freely s-selfdecomposable distributions have atoms, by contrast to the fact that freely selfdecomposable distributions do not have atoms. 

\item 
In case \eqref{a3} a question is if the density of the absolutely continuous part $\mu^{\rm ac}$ is continuous or not. Actually both are possible. An example of continuous $\d\mu^{\rm ac}/\d x$ is given by the free Poisson distribution $\MP^{\boxplus s}$ for $0<s<1$ or by $\cc_p, \frac{1}{2} \leq p <1$ in Example \ref{mixture cauchy}.  An example of discontinuous $\d\mu^{\rm ac}/\d x$ is given by the classical mixture of Boolean stable law $\mathbf{b}_{\alpha, \rho} \circledast \mu$ where $\mu(\{0\})\in(0,1)$ and $(\alpha,\rho)$ satisfies some conditions, see Example \ref{example1} for $\rho=1/2$ and see \cite{AH} for the general case. On the other hand, in case \eqref{a2} there is no example of $\mu$ that has a continuous density, see Conjecture \ref{conjAA} for further details.
\end{enumerate}  
\end{remark}


\begin{proof}
\eqref{a1} Recall that the free L\'evy-Khintchine representation is given by 
\be
\mathcal{C}_\mu(z)= \eta_\mu z+ a_\mu z^2 + \int_{\R}\left( \frac{1}{1- t z} -1 - t z 1_{[-1,1]}(t)\right)\nu_\mu(\d t), 
\ee
and so, for $z=\ri y$, 
\begin{equation}\label{eq0}
\begin{split}
F^{-1}_\mu(z)&= z + z \mathcal{C}_\mu(1/z) \\
&= \ri y \left(- \frac{a_\mu}{y^2}+1- \int_{\R}\frac{t^2}{t^2+y^2}\,\nu_\mu(\d t)\right) + \eta_\mu+  \int_{\R}\left(\frac{t y^2}{t^2+y^2}-t 1_{[-1,1]}(t)\right)\nu_\mu(\d t). 
\end{split}
\end{equation}
If $a_\mu>0$ or $a_\mu=0, \nu_\mu(\R)\in(1,\infty]$ then $\im(F_\mu^{-1}(\ri y))<0$ for some $y>0$ close to 0 by \eqref{eq0}, and so $v_\mu(0)=\inf\{y>0 \mid \im(F^{-1}_\mu(\ri y))>0\} >0$. By Theorem \ref{thm Hu}, $\mu=\mu^{\rm ac}$ and the density of $\mu^{\rm ac}$ is continuous on $\R$ since $\psi_\mu$ is a homeomorphism, $v_\mu$ is continuous on $\R$ and, as we saw, $v_\mu(0)>0$.

\eqref{a2},\eqref{a3} If $a_\mu=0$ and $\nu_\mu(\R)<\infty$ then \eqref{eq0} reduces to  
\be\label{A}
\begin{split}
F^{-1}_\mu(z)
&= \ri y \left(1- \int_{\R}\frac{t^2}{t^2+y^2}\,\nu_\mu(\d t)\right) +c_\mu+ y^2  \int_{\R}\frac{t}{t^2+y^2}\,\nu_\mu(\d t), 
\end{split}
\ee
where $z=\ri y$ and $c_\mu=\eta_\mu-  \int_{\R}t 1_{[-1,1]}(t)\,\nu_\mu(\d t)$. By monotone convergence theorem, the function 
\be
y\mapsto 1- \int_{\R}\frac{t^2}{t^2+y^2}\,\nu_\mu(\d t)
\ee
is a bijection from $(0,\infty)$ onto $(1- \nu_\mu(\R),1)$. Hence if $\nu_\mu(\R)\in[0,1]$ then 
\be
\im (F^{-1}_\mu(\ri y))>0,\qquad y>0,
\ee
 so $v_\mu(0)=0$. Moreover, 
 \be\label{eqA1}
 F^{-1}_\mu(+\ri 0)=\lim_{y\downarrow0}F^{-1}_\mu(\ri y) = c_\mu
 \ee
  by dominated convergence theorem. Furthermore by dominated convergence theorem, 
\be\label{mass}
\begin{split}
\lim_{y\downarrow0}\frac{F^{-1}_\mu(\ri y) -F^{-1}_\mu(+\ri 0)}{\ri y} 
&= \lim_{y\downarrow0}\left(1- \int_{\R}\frac{t^2}{t^2+y^2}\,\nu_\mu(\d t) -  \ri \int_{\R}\frac{t y}{t^2+y^2}\,\nu_\mu(\d t)\right)\\
&= 1- \nu_\mu(\R). 
\end{split}
\ee
 By Theorem \ref{thm Hu},  $\mu=\mu^{\rm ac}$ if $\nu_\mu(\R)=1$, and $\mu$ has an atom at $c_\mu$ with mass $1- \nu_\mu(\R)$ if $\nu_\mu(\R) \in[0,1)$. 
 \end{proof}


 We can prove the following. 

\begin{corollary} \label{cor10}
Suppose that an FID measure $\mu$ is absolutely continuous with respect to the Lebesgue measure. If the density function of $\mu$ is not continuous at 0, then the free L\'evy measure $\nu_\mu$ is a probability measure and $\mu=\MP \boxtimes \nu_\mu$. 
\end{corollary}
\begin{proof}
By Theorem \ref{AC}, the semicircular component $a_\mu$ must be 0 and $\nu_\mu(\R)$ must be $1$. 
This implies that $\mu$ is the shifted compound free Poisson distribution $\delta_{c_\mu}\boxplus(\MP\boxtimes \nu_\mu)$ having the reduced free L\'evy--Khintchine representation \eqref{reducedfLK}. From Huang's formula for density \eqref{density Hu}, the discontinuity point of the density must be $\psi_\mu(0)=F_\mu^{-1}(+\ri 0)$, which is equal to $c_\mu$ from the computation (\ref{eqA1}). Our assumption implies that $c_\mu=0$, so $\mu=\MP \boxtimes \nu_\mu$. 
\end{proof}
From the literature there are many FID distributions that are absolutely continuous with respect to the Lebesgue measure having discontinuous density functions: the standard free Poisson distribution; (mixtures of) Boolean stable laws \cite{AH13b,AH}; some beta distributions of the first and second kinds \cite{H1}; some gamma distributions \cite{H1};  
the square of every symmetric FID random variable having a positive density at 0 \cite[Theorem 2.2]{AHS13}. 
Corollary \ref{cor10} implies that these probability measures are of the form $\MP\boxtimes \nu$ with $\nu(\{0\})=0$. In \cite{AHS13} a stronger result is shown for the last case: a symmetric random variable is FID if and only if its square has the distribution $\MP\boxtimes \sigma$ where $\sigma$ is free regular.

We know that $\mu=\mu^{\rm ac}$ 
in the critical case $a_\mu=0,\nu_\mu(\R)=1$ in Theorem \ref{AC}. Moreover, Corollary \ref{cor10} says that a sufficient condition for $a_\mu=0,\nu_\mu(\R)=1$ is that $\mu=\mu^{\rm ac}$ and its density function $\d \mu/\d x$ is discontinuous at a point. The converse is still open, so let it be a conjecture.   
\begin{conjecture} \label{conjAA}
Let $\nu$ be a probability measure such that $\nu(\{0\})=0$. Then the FID measure $\MP \boxtimes \nu$ is absolutely continuous with respect to the Lebesgue measure and the density is discontinuous at 0. More strongly, the density tends to infinity at 0. 
\end{conjecture}
If this is true then we will get the complete characterization of all FID distributions with discontinuous density without an atom.  

\section{Characterizing symmetric unimodal free L\'evy processes}\label{Unimodal Jurek}

We show the main result $\UUi$, the free analog of Theorem \ref{sym unimodal}.

\begin{theorem}\label{free Jurek unimodal}
Let $\mu$ be symmetric and FID. 
The following statements are equivalent. 
\begin{enumerate}[\rm(1)]\setlength{\itemindent}{3mm}
\item\label{unim} $\mu^{\boxplus s}$ is unimodal for any $s>0$. 
\item\label{fJ} $\mu$ is in $\fU$. 
\end{enumerate}
\end{theorem}
\begin{remark}
There are symmetric unimodal distributions which are not freely s-selfdecomposable. Such examples are given by $\cc_p$ in Example \ref{mixture cauchy} for $p \in [\frac{1}{2}, \frac{1}{2}+ \frac{\sqrt{5}}{10})$ or by Theorem \ref{unimodal-large}. Thus the assertion \eqref{unim} in Theorem \ref{free Jurek unimodal} is not equivalent to ``$\mu^{\boxplus s}$ is unimodal for some $s>0$.''
\end{remark}
Thus, if $\mu$ is symmetric and ID then we have the equivalence 
\begin{center}
$\mu^{\ast s}$ is unimodal for all $s>0$ $\Longleftrightarrow$ $\Lambda(\mu)^{\boxplus s}$ is unimodal for all $s>0$.  
\end{center}

The easier part of the proof is \eqref{unim}$\Rightarrow$\eqref{fJ} and it follows from the following lemma.  

\begin{lemma}\label{convergence fLM} Let $\mu$ be FID, and let $\nu$ be its free L\'evy measure. Then it holds that 
$$
\int_{\R}f(x)\,\nu(\d x) =  \lim_{t\downarrow 0} \frac{1}{t}\int_{\R}f(x)\,\mu^{\boxplus t}(\d x)
$$
for any bounded continuous function $f$ on $\R$ which is zero in a neighborhood of $0$. 
\end{lemma}
\begin{proof} 
This is follows from \cite[Theorem 5.10]{BV93} since we have \eqref{fLM}.  
\end{proof}
\begin{proof}[Proof of Theorem \ref{free Jurek unimodal} \eqref{unim}$\Rightarrow$\eqref{fJ}] 
Note that a symmetric (possibly infinite) measure $\rho$ on $\R$ is unimodal if and only if the distribution function $D_\rho(x):=\rho((-\infty,x])$ is convex on $(-\infty,0)$, i.e.\ $D_\rho(p x + (1-p) y) \leq p D_\rho(x)+(1-p)D_\rho(y)$ for all $p\in(0,1), x,y\in(-\infty,0)$. 

Let $\nu$ be the free L\'evy measure of $\mu$. The convergence in Lemma \ref{convergence fLM} implies that the functions $D_n(x):=D_{n \mu^{\boxplus 1/n}}(x)$ converge as $n\to\infty$ to $D_\nu(x)$ at all points $x <0$ where $D_\nu$ is continuous. Let $x,y <0$ be continuous points of $D_\nu$. Since $D_n$ is convex, by taking the limit  we have 
\be\label{convex}
D_\nu(p x + (1-p) y) \leq p D_\nu(x)+(1-p)D_\nu(y)
\ee
 where $p$ is taken so that $D_\nu$ is continuous at $p x +(1-p) y$. Such $p$'s are dense in $(0,1)$ and then by the right continuity of $D_\nu$ \eqref{convex} holds for all $p\in(0,1)$. Again by right continuity \eqref{convex} holds for all $x,y <0$. This implies that $D_\nu$ is convex, and hence $\nu$ is unimodal. 
\end{proof}

The converse part \eqref{fJ}$\Rightarrow$\eqref{unim} requires more efforts. 
Let $\nu$ be the free L\'evy measure of a probability measure $\mu$. Consider the function $A_\nu\colon\C^+\cup\R\to [0,\infty]$ defined by 
\be
A_\nu(x+\ri y)= \int_{\R}\frac{t^2\nu(t)}{(x-t)^2+y^2}\6t. 
\ee
This function is important since, if $\mu$ has no semicircular component, 
\be\label{eqAA}
s y\left(\frac{1}{s}- A_\nu (x+\ri y)\right) = \im\!\left(F_{\mu^{\boxplus s}}^{-1}(x+\ri y)\right)
\ee 
for $x\in\R, y>0$ (it is easy to extend the definition of $A_\nu$ when $\mu$ has a semicircular component so that $\eqref{eqAA}$ holds, but for simplicity we will avoid such a case). The function $A_\nu$ was denoted by $F_k$ in \cite{HT}. 

\begin{lemma}\label{lemma-density} 
Let $\mu$ be an FID distribution and $(\eta_\mu,a_\mu,\nu_\mu)$ be its free characteristic triplet. Suppose that  $a_\mu=0$, $\nu_\mu(\R)=\infty$ and $s >0$. If the equation $A_\nu(R\sin(\theta)e^{\ri \theta})=\frac{1}{s}$ has at most two solutions $\theta \in(0,\pi)$ for each fixed $R\in(0,\infty)$, then $\mu^{\boxplus s}$ is unimodal.  
\end{lemma} 
\begin{proof}
The proof is similar to \cite[Proposition 3.8]{HT}. Let us denote $v_s:=v_{\mu^{\boxplus s}}$ and $\psi_s:=\psi_{\mu^{\boxplus s}}$. 
The assumptions imply that $v_s(0)>0$ and $\mu^{\boxplus s}$ is absolutely continuous with respect to the Lebesgue measure and the density $f_s(x):=\d \mu^{\boxplus s} /\d x$ is continuous on $\R$ by Theorem \ref{AC}.  

We first show that for each $\rho \in(0,\infty)$, there are at most two solutions $x$ to the equation 
\be\label{eq N}
\rho = f_s(\psi_s(x))= \frac{v_s(x)}{\pi(x^2 + v_s(x)^2)}.  
\ee
It then suffices to consider $x$ such that $v_s(x)>0$, and for such $x$, $v_s(x)$ is the unique solution $y>0$ to the equation 
\be\label{eq M}
A_{\nu}(x+\ri y)=\frac{1}{s}. 
\ee
The curve $\{x+\ri y \in\C^+\mid \frac{y}{\pi(x^2+ y^2)} = \rho\}$ can be expressed as $\{\frac{1}{\pi\rho} \sin(\theta) \e^{\ri \theta}\mid \theta \in(0,\pi)\}$ in polar coordinates, which is a punctured circle tangent to the $x$ axis at 0. By \eqref{eq N} and \eqref{eq M} it suffices to show that for each $R>0$ there are at most two solutions $\theta\in(0,\pi)$ to the equation 
\be
A_\nu\left(\frac{1}{\pi \rho}\sin(\theta)e^{\ri \theta}\right)=\frac{1}{s}, 
\ee
which is the case by assumption. 

By Proposition \ref{tail} and the continuity of $f_s$, the density $f_s$ attains the global maximum at a point $x_0$. If $ \mu^{\boxplus s}$ were not unimodal then the density would attain a local maximum at a point $x_1\neq x_0$ (since $f_s$ is real analytic and hence it does not have a plateau). By intermediate value theorem there exists $c <f_s(x_1)$ such that the equation $f_s(x)=c$ has at least four solutions, a contradiction. 
\end{proof}

\begin{lemma}\label{key-lemma}
Let $\mu$ be symmetric and FID, and let $\nu$ be its free L\'evy measure. Assume that $\nu$ is of the form $\nu(\d t)= \ell(\vert t\rvert) 1_{\R\setminus\{0\}}(t)\6t$, where $\ell \colon(0,\infty)\to [0,\infty)$ is
a function that satisfies the following conditions:

\begin{enumerate}[\rm(a)]\setlength{\itemindent}{3mm}
\item\label{ella} $\ell\neq0$, $\ell\in C^2((0,\infty))$ and $\ell'\leq 0$; 
\item $\lim_{t\downarrow 0} t^3\, \ell(t)=0$; 
\item\label{ellc} There exists $M>0$ such that $\ell(t)=0$ for $t>M$. 
\end{enumerate}
Then for any $R\in(0,\infty)$ the function
\[
\theta\mapsto
A_\nu(R \sin(\theta)\e^{\ri\theta})
\]
is strictly decreasing on $(0,\frac{\pi}{2}]$ and
strictly increasing on $[\frac{\pi}{2},\pi)$.
\end{lemma}

\begin{proof}
Let $u$ be a new variable defined by $t=(R\sin\theta)u$. Then  
\be
A_\nu(R\sin(\theta)\e^{\ri\theta})= R \sin\theta\int_{\R}\frac{u^2\ell(R u \sin\theta)}{1-2u\cos\theta +u^2}\6u,\qquad \theta\in(0,\pi). 
\ee
Let $h(u):=\ell(R u)$, $u>0$, and  
\begin{align}
&\xi(x):=\int_0^\infty \frac{u^2\sqrt{1-x^2}}{1-2xu+u^2}\,h(u\sqrt{1-x^2})\6u,& x\in(-1,1), \\
&\Xi(x):=\xi(x)+\xi(-x),& x\in(-1,1). 
\end{align}
Since $\nu$ is symmetric, we have
\be\label{eq+-}
A_\nu(R\sin(\theta)\e^{\ri\theta})=\Xi(\cos\theta), \qquad\theta\in(-\pi,\pi).
\ee
Since $\Xi$ is symmetric, it suffices to show that 
\be\label{goal}
\Xi'(x)>0,\qquad x\in(0,1). 
\ee
For any $x$ in $(-1,1)$ we note first by differentiation under the
integral sign that
\be\label{eq1}
\begin{split}
\xi'(x)
&= \int_0^\infty\left(-\frac{x u^2}{u^2-2x u+1}+\frac{2 u^3 (1-x^2)}{\left(u^2-2x u+1\right)^2}\right)\frac{1}{\sqrt{1-x^2}}h(u\sqrt{1-x^2})\6u \\
&\quad+\int_0^\infty\frac{x u^3}{u^2-2x u+1}(-h'(u\sqrt{1-x^2}))\6u. 
\end{split}
\ee
We can prove the following: 
\be
\begin{split}
k(x,u)&:=\int_0^u\left(-\frac{x t^2}{t^2-2  x t +1}+\frac{2 t^3(1-x^2)}{\left(t^2-2  x t+1\right)^2}\right)\d t \\
&= - x u+\frac{4 x^3 u -3  x u -2 x^2+1}{u^2-2x u+1}+4 x \sqrt{1-x^2} \arctan\left(\frac{u-x}{\sqrt{1-x^2}}\right) \\
&\quad+\left(1-2 x^2\right) \log(u^2-2  x u +1)  -1 +2x^2+ 4 x \sqrt{1-x^2}\arctan\left(\frac{x}{\sqrt{1-x^2}}\right). 
\end{split}
\ee
By integration by parts,  \eqref{eq1} becomes 
\be
\begin{split}
\xi'(x)
&= \int_0^\infty K(x,u)\,(-h'(u\sqrt{1-x^2}))\6u, 
\end{split}
\ee
where
\be
\begin{split}
K(x,u)&:= k(x,u)+\frac{x u^3}{u^2-2x u+1} \\
&= \frac{4x^2u^2 -u^2  -2 x u }{u^2-2x u+1}+\left(1-2 x^2\right) \log(u^2-2  x u +1)\\
&~~~+4 x \sqrt{1-x^2}\left( \arctan\left(\frac{u-x}{\sqrt{1-x^2}}\right) +\arctan\left(\frac{x}{\sqrt{1-x^2}}\right)\right). 
\end{split}
\ee 
Therefore  
\be
\Xi'(x)=\int_0^\infty L(x,u)(-h'(u\sqrt{1-x^2}))\6u, 
\ee
where
\be 
L(x,u):=K(x,u)-K(-x,u).  
\ee
In order to show \eqref{goal} it suffices to show that $L(x,u)>0$ for $(x,u)\in(0,1)\times (0,\infty)$. For this we compute the derivative 
\be
\frac{\partial}{\partial u} L(x,u) = \frac{4 u^2 x \left( \left(5-8 x^2\right)u^4+ 2\left(3-2 x^2\right)u^2+1\right)}{(u^2-2x u +1)^2(u^2+2x u +1)^2}. 
\ee
By calculus, for $x\in(0,\frac{\sqrt{10}}{4}]$, the map $u\mapsto L(x,u)$ is strictly increasing in $(0,\infty)$. For $x\in(\frac{\sqrt{10}}{4},1)$, there exists a unique $\alpha(x)\in(0,\infty)$ such that the map $u\mapsto L(x,u)$ is strictly increasing in $(0,\alpha(x))$ and strictly decreasing in $(\alpha(x),\infty)$. 
Since $L(x,0)=0$ and $\lim_{u\to\infty} L(x,u)= 4\pi x\sqrt{1-x^2}>0$ for $x\in(0,1)$, we then conclude that $L(x,u)>0$ for $(x,u)\in(0,1)\times (0,\infty)$. 
\end{proof}


\begin{proof}[Proof of Theorem \ref{free Jurek unimodal} \eqref{fJ}$\Rightarrow$\eqref{unim}] 
We first assume that $\mu \in \fU$ has the free characteristic triplet $(\eta, 0, \nu)$ and the free L\'evy measure $\nu$ satisfies the assumptions of Lemma \ref{key-lemma}, and moreover,  
 \be\label{infty}
\nu(\R)=\infty.  
\ee 
By Lemma \ref{key-lemma}, for each $R>0$ the function $\theta \mapsto A_\nu(R \sin (\theta)\e^{\ri \theta})$ has at most two solutions $\theta\in(0,\pi)$, and so by Lemma \ref{lemma-density}, $\mu^{\boxplus s}$ is unimodal. 

A general symmetric $\mu \in \fU$ with free characteristic triplet $(\eta,a,\nu)$ can be approximated by the probability measures considered above. The arguments are similar to \cite[Lemma 6]{HT} so only the sketch is given here.   The free L\'evy measure $\nu$ is of the form $\ell(\vert t \rvert)\,\d t$ where $\ell\colon (0,\infty)\to [0,\infty)$ is non-increasing. Then we define 
\be
\ell_n^0(t):= 
\begin{cases} 
\ell(\frac{1}{n}), & 0<t<\frac{1}{n}, \\
\ell(t), & \frac{1}{n} \leq t \leq n, \\
0,& t >n.    
\end{cases}
\ee
Then $\ell_n^0 \leq \ell_{n+1}^0, n\in\N$. Take a nonnegative function $\varphi \in C^\infty(\R)$ such that $\supp(\varphi) \subset [-1,0]$ and $\int_{-1}^0 \varphi(t)\,\d t=1$. Define $\varphi_n(t):= n \varphi(n t)$ and $\ell_n:= (\varphi_n \ast \ell_n^0)\vert_{(0,\infty)}$. We can show that $\supp(\ell_n) \subset (0,n]$, $\ell_n \leq \ell_{n+1}$ and $\ell_n(t) \uparrow \ell(t)$ at almost all $t \in(0,\infty)$ (with respect to the Lebesgue measure). Finally take a nonnegative function $\rho \in C^\infty(\R)$ such that $\rho(-t)=\rho(t)$, $\rho$ is strictly positive in a neighborhood of 0, $\rho'(t)\leq0$ for $t\in(0,\infty)$, $\supp(\rho)$ is compact  and $\int_\R \rho(t)\,\d t=1$. We define 
\be
\nu_n(\d t):= \ell_n(\vert t\rvert)\,\d t + \frac{a + n^{-1}}{t^2} n\rho(n t)\, \d t. 
\ee
Let $\mu_n$ be the FID distribution having the free characteristic triplet $(\eta, 0,\nu_n)$. The free L\'evy measure $\nu_n$ satisfies the assumptions of Lemma \ref{key-lemma} and \eqref{infty}, so $\mu_n^{\boxplus s}$ is unimodal for all $s>0$. One can show the weak convergence  
\be
\frac{t^2}{1+t^2}\, \nu_n(\d t) \to \frac{t^2}{1+t^2}\, \nu(\d t) + a \delta_0, 
\ee
so by \cite[Theorem 3.8]{B-NT02} $\mu_n^{\boxplus s}$ converges to $\mu^{\boxplus s}$ weakly for each $s>0$. Since $\mu_n^{\boxplus s}$ is unimodal and the weak convergence preserves the unimodality, $\mu^{\boxplus s}$ is unimodal. 
\end{proof}



Up to now there is no counterexample to: 
\begin{conjecture}
Let $\mu$ be an ID distribution. The following are equivalent: 
\begin{enumerate}[\rm(1)]\setlength{\itemindent}{3mm}
\item $\mu^{\ast s}$ is unimodal for any $s>0$; 
\item $\Lambda(\mu)^{\boxplus s}$ is unimodal for any $s>0$.  
\end{enumerate} 
\end{conjecture}


In the classical case if we drop the assumption of symmetry, then Theorem \ref{sym unimodal} fails to hold as Proposition \ref{Wolfe ex} shows, so there are non-unimodal probability measures in the Jurek class $\cU$. The free analog is not known. 
 \begin{conjecture}
There exists a non-unimodal probability measure in $\fU$. 
\end{conjecture}


Examples of freely s-selfdecomposable or selfdecomposable probability measures are provided below.

\begin{example}\label{mixture cauchy} Let $\cc_p$ be a mixture of a Cauchy distribution and $\delta_0$: 
\be
\cc_p(\d x)=p\delta_0+\frac{1-p}{\pi(1+x^2)}1_{\R}(x)\6x,\qquad p\in[0,1]. 
\ee
This measure is symmetric and unimodal. It was proved in \cite[Proposition 5.8]{AH} that $\cc_p$ is FID if and only if $p\in\{0\}\cup [\frac{1}{2},1]$. 
Moreover, we claim here that: 
\begin{enumerate}[\rm(1)]\setlength{\itemindent}{3mm}
\item\label{Ex fSD} $\cc_p$ is in $\fSD$ if and only if $p=0,1$; 
\item\label{Ex fU} $\cc_p$ is in $\fU$ if and only if $p\in\{0\}\cup [\frac{5+\sqrt{5}}{10},1]$. 
\end{enumerate}
The first point \eqref{Ex fSD} is easier. By Remark \ref{rem0}\eqref{rem0-1}, $\cc_p$ is not in $\fSD$ for $0<p<1$. The Cauchy distribution $\cc_0$ has the free L\'evy measure $\pi^{-1}t^{-2}\,\d t$ and the delta measure $\cc_1$ has the free L\'evy measure 0, so $\cc_0,\cc_1 \in\fSD$. 

The second point \eqref{Ex fU} is more delicate and needs a lot of computation. Let $p\in[\frac{1}{2},1]$, for which $\cc_p$ is FID. The Voiculescu transform is given by 
\be
\varphi_{\cc_p}(z)=\frac{1}{2}(-z-\ri+\sqrt{z^2+2(2p-1)\ri z-1}),\qquad z\in\C^+, 
\ee
where the map $z\mapsto\sqrt{z^2+2(2p-1)\ri z-1}$ is defined analytically in $\C^+$ so that it preserves the set $\ri(0,\infty)$.  
The Stieltjes inversion formula implies that the free L\'evy measure is given by $-\frac{1}{\pi x^2}\lim_{y\downarrow0}\im(\varphi_{\cc_p}(x+\ri y))\6x$. We put $re^{\ri \theta}=(x+\ri 0)^2+2(2p-1)\ri (x+\ri 0)-1$, $\theta\in(-\pi/2,3\pi/2)$. Note then that $r=\sqrt{x^4+u x^2+1}$, where $u:=2(8p^2-8p+1)\in[-2,2]$.  Then $\lim_{y\downarrow0}\im(\varphi_{\cc_p}(x+\ri y))\6x = \sqrt{r}\sin\frac{\theta}{2}$, and so we get the free L\'evy measure 
\be
\nu_{\cc_p}(\d x)=\frac{1}{2\sqrt{2}\pi x^2}\left(\sqrt{2}-(r -x^2+1)^{1/2}\right)1_{\R\setminus\{0\}}(x)\6x, 
\ee
which is symmetric. With the new variable $y=x^2$, the density of $\nu_{\cc_p}$ reads 
\be\label{eq99}
\frac{1}{2\sqrt{2}\pi y}\left(\sqrt{2}-\sqrt{\sqrt{y^2 +u y +1} -y +1}\right), \qquad y\geq0. 
\ee
After a lot of calculation, we can show that the function \eqref{eq99} is non-increasing on $(0,\infty)$ if and only if $u\in[-\frac{6}{5},2]$, which is equivalent to $p\in[\frac{5+\sqrt{5}}{10},1]$. 
\end{example}

\begin{example}\label{example1}
 \begin{enumerate}[\rm(1)]
 \item The free Meixner distribution $\mathbf{fm}_{a,b}$ for $a\in\R, b\geq-1$ (see \cite{SY01,A03}) is defined by 
\begin{align}
&\mathbf{fm}_{a,b}(\d x)= \frac{\sqrt{4(1+b)-(x-a)^2}}{2\pi(b x^2+a x +1)}1_{[a-2\sqrt{1+b}, a+2\sqrt{1+b}]}(x)\6x +\text{0,1 or 2 atoms},  \label{fm density}\\
&G_{\mathbf{fm}_{a,b}}(z)= \frac{(1+2b)z+a -\sqrt{(z-a)^2-4(1+b)}}{2(bz^2 +a z+1)}, \label{fm cauchy}\\
&\varphi_{\mathbf{fm}_{a,b}}(z)=\frac{-a+z-\sqrt{(z-a)^2-4b}}{2b},  
 \end{align}
where $\sqrt{w}$ is continuously defined in $\C\setminus [0,\infty)$. It is known that $\mathbf{fm}_{a,b}$ is FID if and only if $b\geq0$ \cite{SY01}. For $b>0$, its free L\'evy measure is given by 
 \be
 \nu_{\mathbf{fm}_b}(\d x)= \frac{\sqrt{4b-(x-a)^2}}{2\pi b x^2}1_{[a-2\sqrt{b}, a+2\sqrt{b}]\setminus\{0\}}(x)\6x. 
 \ee
By elementary calculus, we have the following equivalence for $a\in\R, b\geq0$: 
\begin{center} 
 
$\mathbf{fm}_{a,b}\in \fSD$ $\Longleftrightarrow$ $\mathbf{fm}_{a,b}\in \fU$  $\Longleftrightarrow$ $ a^2 \leq 4b$.   
\end{center} 
 \item The free stable laws are freely selfdecomposable. Let $\mathfrak{A}$ be the set of admissible pairs 
 \be
\mathfrak{A}=\{(\alpha,\rho)\mid \alpha \in(0,2], \rho \in [0,1] \cap[1-\alpha^{-1}, \alpha^{-1}]\}. 
\ee
Assume that $(\alpha,\rho)$ is admissible. Let $\fst_{\alpha,\rho}$ be the free stable law \cite{BV93,BP99} characterized by the following. 
\begin{enumerate}[\rm(i)] 
\item If $\alpha \neq1$, then 
$$
\varphi_{\fst_{\alpha,\rho}}(z)= - e^{i \rho\alpha\pi}z^{1-\alpha}, \qquad z\in \C^+. 
$$
\item If $\alpha =1$, then 
$$
\varphi_{\fst_{1,\rho}}(z)=-2\rho i +\frac{2(2\rho-1)}{\pi}\log z, \qquad z\in \C^+.  
$$
\end{enumerate}
Then the free L\'evy measures are as follows. 
 \begin{enumerate}[\rm(i)]
\item If $\alpha \neq1$, then 
$$ 
\nu_{\fst_{\alpha,\rho}}=\frac{\sin(\alpha(1-\rho)\pi)}{|x|^{1+\alpha}}1_{(-\infty,0)}(x)\6x +  \frac{\sin(\alpha\rho\pi)}{x^{1+\alpha}}1_{(0,\infty)}(x)\6x. 
$$
\item 
 If $\alpha =1$, then 
$$ 
\nu_{\fst_{1,\rho}}=\frac{2(1-\rho)}{x^{2}}1_{(-\infty,0)}(x)\6x +  \frac{2\rho}{x^{2}}1_{(0,\infty)}(x)\6x. 
$$
\end{enumerate}
\item Let $\mathbf{b}_{\alpha,\rm sym}, \alpha\in(0,2)$ be the symmetric Boolean stable law defined by  
\begin{align}
&\frac{\d\mathbf{b}_{\alpha,\rm sym}}{\d x} = \frac{\sin(\frac{ \alpha\pi}{2})}{\pi} \frac{|x|^{\alpha-1}}{|x|^{2\alpha}+2|x|^\alpha \cos(\frac{ \alpha\pi}{2})+1},& x \in \R, \\
& G_{\mathbf{b}_{\alpha,\rm sym}}(z)= \frac{1}{z+\ri (-\ri z )^{1-\alpha}},& z\in\C^+, 
\end{align}
where for $\alpha\neq1$ the map $z\mapsto (-\ri z)^{1-\alpha}$ is defined analytically in $\C^+$ so that it maps $\ri(0,\infty)$ onto itself. 
For any probability measure $\mu$ on $[0,\infty)$, the classical mixture $\mathbf{b}_{\alpha,\rm sym} \circledast \mu$ is known to be in $\fU$ for $\alpha \in (0,\frac{1}{2}]$ \cite[Theorem 5.13(2)]{AH}, but it is not freely selfdecomposable unless $\mu=\delta_0$.  
If we denote by $\mu^{p}$ the induced measure by the map $x\mapsto x^p$, the free L\'evy measure of  $\mathbf{b}_{\alpha,\rm sym} \circledast \mu^{1/\alpha}$ is given by 
\be
\nu_{\mathbf{b}_{\alpha,\rm sym} \circledast \mu^{1/\alpha}}=\bm{\pi}^{\boxtimes \frac{1-2\alpha}{\alpha}} \boxtimes \mathbf{f}_{\alpha,1/2} \boxtimes \mu^{\boxtimes 1/\alpha}, 
\ee
see \cite[Proposition 4.21(1)]{AH}. Mixtures of some positive Boolean stable laws are also in $\fU$, see \cite[Proposition 4.21(2)]{AH}. 

\item The probability measure $\bm{\pi} \boxtimes \mathbf{m}_{\alpha,\rho}$ was investigated in \cite{AH13}, where $\mathbf{m}_{\alpha,\rho}$ is a monotone stable law. We restrict the parameters to $(\alpha,\rho) \in (0,1)\times [0,1]$, and then the monotone stable law is defined by 
\be
F_{\mathbf{m}_{\alpha,\rho}}(z)=(z^\alpha+e^{\ri \alpha \rho \pi})^{1/\alpha},\qquad z\in\C^+,  
\ee
where all the powers are the principal value. 
Since $\mathbf{m}_{\alpha,\rho}$ is unimodal with mode 0 for $\rho\in [\frac{\alpha\pi}{1+\alpha}, \frac{\pi}{1+\alpha}]$ \cite{HS}, and since the free L\'evy measure of $\bm{\pi} \boxtimes \mathbf{m}_{\alpha,\rho}$ is $\mathbf{m}_{\alpha,\rho}$, so the measure $\bm{\pi} \boxtimes \mathbf{m}_{\alpha,\rho}$ is in $\fU$ for $\alpha\in(0,1), \rho\in [\frac{\alpha\pi}{1+\alpha}, \frac{\pi}{1+\alpha}]$.


\item The Student t-distribution with 3 degrees of freedom 
\begin{align}
&\mathbf{St}_{3}(\d x) = \frac{2}{\pi(1+x^2)^2}1_{\R}(x)\6x,  \label{student density}\\
&G_{\mathbf{St}_{3}}(z)= \frac{z+2\ri}{z^2+2\ri z-1}, \label{student cauchy}\\
&\varphi_{\mathbf{St}_{3}}(z)=\frac{-z-2\ri+ \sqrt{z^2 + 4\ri z}}{2}, 
 \end{align}
is FID, where the map $z\mapsto \sqrt{z^2+4\ri z}$ is defined in $\C^+$ so that it preserves $\ri (0,\infty)$. 
The free L\'evy measure can be written as 
\be
\nu_{\mathbf{St}_{3}}(\d x)= \frac{1}{2\sqrt{2}\pi x^2}\left(2\sqrt{2}- \sqrt{\sqrt{x^4+16 x^2}-x^2}\right)1_{\R\setminus\{0\}}(x)\6x. 
\ee
One can easily show that the function $2\sqrt{2}- (\sqrt{x^4+16 x^2}-x^2)^{1/2}$ is decreasing on $(0,\infty)$, and hence 
the density function of the free L\'evy measure is of the form $j(x)/x^2$, where $j(x)$ is increasing on $(-\infty,0)$ and decreasing on $(0,\infty)$. This in particular implies that $\mathbf{St}_{3}$ is freely selfdecomposable. 
\end{enumerate}
\end{example}





\section{Unimodality of free L\'evy processes with boundedly supported L\'evy measure in large time} \label{Unimodal Large}
The main result of this section is $\UUii$ which does not have a classical analog. 
\begin{theorem}\label{unimodal-large}
Let $\mu$ be an FID measure whose free L\'evy measure $\nu$ satisfies $\supp(\nu)\subset[-M,M]$ for some $M>0$. Suppose that $\mu$ is not a point measure. 
Then $\mu^{\boxplus s}$ is unimodal for $s \geq \frac{4 M^2}{\sigma^2(\mu)}$. The constant $4$ is optimal. 
 \end{theorem}
\begin{remark}
Note that $\nu$ has a bounded support if and only if $\mu$ has a bounded support.   
\end{remark}
 
 
 
  \begin{proof}
{\bf Step 1} (Basic calculation for good free L\'evy measures). Let $(\eta, a, \nu)$ be the free characteristic triplet of $\mu$. Assume that $\mu$ does not have a semicircular component (i.e.\ $a=0$) and that $\nu(\d t)$ is of the form $\frac{k(t)}{|t|}$\d t, where $k \in C^\infty(-\infty,\infty)$,  $\supp(k) \subset [-M,M]$ and $k>0$ in a neighborhood of $0$. Then $\mu$ is absolutely continuous with respect to the Lebesgue measure and the probability density function is continuous on $\R$ by Theorem \ref{AC} since now $\nu(\R)=\infty$. 

Let $u$ be a new variable defined by $t=(R \sin\theta)u$. Then  
\be
A_\nu(R\sin(\theta)\e^{\ri\theta}) = \int_{\R}\frac{|u| k(R u \sin\theta)}{1-2u\cos\theta +u^2}\6u,\qquad \theta\in(0,\pi). 
\ee
Take any nonzero function $h\colon(0,\infty)\to[0,\infty)$ from
$C^2(0,\infty)$, supported on $(0,M/R]$ having bounded derivatives $h,h',h''$ on $(0,\infty)$. Then let
\be
\xi_h(x):=\int_0^\infty \frac{u}{1-2xu+u^2}\,h(u\sqrt{1-x^2})\6u,\qquad x\in(-1,1). 
\ee
Note then that if we define $k_R ^\pm(u):=k(\pm R u)$ for $u>0$, and
\be
\Xi_R(x)=\xi_{k_R^+}(x)+\xi_{k_R^-}(-x), \qquad x\in(-1,1),
\ee
then it holds that
\be\label{eq+-}
A_\nu(R\sin(\theta)\e^{\ri\theta})=\Xi_R(\cos\theta), \qquad \theta\in(-\pi,\pi).
\ee
For any $x$ in $(-1,1)$ we note first by differentiation under the
integral sign that
\be\label{0d}
\begin{split}
\xi_h'(x)
&= \int_0^\infty \frac{2u^2}{(1-2ux+u^2)^2}\,h(u\sqrt{1-x^2})\6u \\
&~~~-\int_0^\infty \frac{u^2}{1-2ux+u^2}\cdot\frac{x}{\sqrt{1-x^2}}\,
h'(u\sqrt{1-x^2})\6u,  
\end{split}
\ee
and by integration by parts,  
\be\label{1d}
\begin{split}
\xi_h'(x)
&= \int_0^\infty \frac{2u^2}{(1-2ux+u^2)^2}\,h(u\sqrt{1-x^2})\6u
\\
&~~~+\int_0^\infty \frac{\partial}{\partial u}
\left(\frac{u^2}{1-2ux+u^2}\right)\cdot\frac{x}{1-x^2}\,h(u\sqrt{1-x^2})\6u
\\
&=\int_0^\infty\frac{2u(x+(1-2x^2)u)}{(1-2xu+u^2)^2(1-x^2)}\,h(u\sqrt{1-x^2})\6u. 
\end{split}
\ee
Using Leibniz' formula and integration by parts, we
find that
\be\label{2d}
\begin{split}
\xi_h''(x) = &\int_0^\infty\frac{\partial}{\partial x}\left(\frac{2u((1-2x^2)u+x)}{(1-2xu+u^2)^2(1-x^2)}\right) h(u\sqrt{1-x^2})\6u \\
&\quad+\int_0^\infty\frac{2u(x+(1-2x^2)u)}{(1-2xu+u^2)^2(1-x^2)} \frac{(-x u)}{\sqrt{1-x^2}}h'(u\sqrt{1-x^2})\6u\\ 
&=\int_0^\infty\frac{\partial}{\partial x}\left(\frac{2u(x+(1-2x^2)u)}{(1-2xu+u^2)^2(1-x^2)}\right) h(u\sqrt{1-x^2})\6u \\
&\quad+\int_0^\infty\frac{2u(x+(1-2x^2)u)}{(1-2xu+u^2)^2(1-x^2)^2}(-x u) \frac{\d}{\d u} h(u\sqrt{1-x^2})\6u\\ 
&=\int_0^\infty \frac{2u P(u,x)}{\left(1-2 u x+u^2\right)^3(1-x^2)^2}h(u\sqrt{1-x^2})\6u, 
\end{split}
\ee
where 
\be
P(u,x)=1+5 u^2+3 u x-3 u^3 x+3 x^2-11 u^2 x^2-12 u x^3+2 u^3 x^3+12 u^2 x^4. 
\ee
By putting $x=\cos \theta$, we will show that 
\begin{center} 
if $R\in(0,\infty)$ and $s \geq \frac{4 M^2}{\sigma^2(\mu)}=: T$ then the equation $\Xi_R(x)=\frac{1}{s}$ has at most two solutions $x\in(-1,1)$ 
\end{center}
through Steps 2--4 below.

\vspace{10pt}
\noindent
{\bf Step 2}. We will show that if $s \geq T$ and $0<R <\sqrt{2}M$ then the equation $\Xi_R(x)= \frac{1}{s}$ does not have a solution $\theta\in(0,\pi)$. If $0<R <\sqrt{2}M, t \in [-M,M]$ and $s\geq T$ then it holds that 
\begin{equation}
\begin{split}
R^2 \sin^2\theta - 2 R t \sin\theta\cos\theta +t^2 
&= \frac{R^2}{2}+t^2 - \left(\frac{R^2}{2}\cos (2\theta) +R t \sin(2\theta)\right)  \\
&\leq M^2+ M^2 + \sqrt{\frac{R^4}{4}+R^2 t^2} \\
&\leq 2 M^2 + \sqrt{3}M^2 < 4 M^2, 
\end{split}
\end{equation}
and so we obtain 
\begin{equation}
\begin{split}
\Xi_R(x)=A_\nu(R\sin(\theta)e^{\ri \theta}) 
&= \int_{-M}^M \frac{t^2 \,\nu(\d t)}{R^2 \sin^2\theta - 2R t \sin\theta\cos\theta +t^2}\\ 
&> \frac{\sigma^2(\mu)}{4M^2}=T^{-1} \geq \frac{1}{s}.  
\end{split}
\end{equation}
Note that $\int_{-M}^M t^2\,\nu(\d t)=\sigma^2(\mu)$. Thus the proof is finished. 

We may thus assume that $R \geq\sqrt{2}M$ hereafter.

\vspace{10pt}
\noindent
{\bf Step 3}. We will show that if $R \geq\sqrt{2}M$ and $s\geq T$ then the equation $\Xi_R(x)=\frac{1}{s}$ does not have a solution $x\in (-1,-(1- \frac{M^2}{2 R^2}))\cup (1- \frac{M^2}{2 R^2},1)$. 
Recalling that $x=\cos\theta$, for $|x|>1- \frac{M^2}{2 R^2}$ we have that $\sin^2\theta < \frac{M^2}{R^2} -\frac{M^4}{4R^4} < \frac{M^2}{R^2}$ and 
we have the estimate 
\begin{equation}
\begin{split}
A_\nu(R\sin(\theta)e^{\ri \theta}) 
&= \int_{-M}^M \frac{t^2\, \nu(\d t)}{R^2 \sin^2\theta - 2 R t  \sin\theta \cos\theta +t^2}\\ 
&> \int_{-M}^M \frac{t^2\, \nu(\d t)}{R^2 \cdot \frac{M^2}{R^2} + 2 R M \sqrt{\frac{M^2}{R^2}}+M^2  }\\
&= \frac{1}{4 M^2} \int_{-M}^M t^2\, \nu(\d t) = T^{-1} \geq \frac{1}{s}
\end{split}
\end{equation}
and so the proof is finished. 


\vspace{10pt}
\noindent
{\bf Step 4}. We show that if $R \geq \sqrt{2}M$ then the equation $\Xi_R(x)=\frac{1}{s}$ considered in $[-(1- \frac{ M^2}{2 R^2}),1- \frac{M^2}{2 R^2}]$ has at most two solutions $x$. For this it suffices to show that there exists $x_0=x_0(\mu,R) \in(-\frac{\sqrt{2}}{2},\frac{\sqrt{2}}{2})$ such that $\Xi_R$ is decreasing on $[-(1- \frac{ M^2}{2 R^2}),x_0)$ and increasing on $(x_0,1- \frac{M^2}{2 R^2}]$. Then it suffices to show that   
\begin{itemize} 
\item $\Xi_R '(x) = \xi_{k_R^+}'(x)-\xi_{k_R^-}'(-x)<0$ for $-(1- \frac{M^2}{2 R^2})\leq x \leq -\frac{\sqrt{2}}{2}$,  
\item $\Xi_R ''(x) = \xi_{k_R^+}''(x)+\xi_{k_R^-}'(-x)>0$ for $-\frac{\sqrt{2}}{2}\leq x \leq \frac{\sqrt{2}}{2}$, 
\item $\Xi_R '(x) = \xi_{k_R^+}'(x)-\xi_{k_R^-}'(-x)>0$ for $\frac{\sqrt{2}}{2}\leq x \leq 1- \frac{M^2}{2 R^2}$.  
\end{itemize}
Furthermore, it suffices to show that 
\begin{enumerate}[\rm(1)]\setlength{\itemindent}{3mm}
\item\label{1} $\xi_h'(x)<0$ for $x\in [-(1-\frac{M^2}{2 R^2}), -\frac{\sqrt{2}}{2}]$,
\item\label{3} $\xi_h''(x)>0$ for $x\in[-\frac{\sqrt{2}}{2},\frac{\sqrt{2}}{2}]$, 
\item\label{2} $\xi_h'(x)>0$ for $x\in [\frac{\sqrt{2}}{2},1- \frac{ M^2}{2 R^2}]$. 
\end{enumerate}

\eqref{1} This follows from (\ref{1d}).

\eqref{3} We want to use the expression \eqref{2d}. Recalling that $R\geq \sqrt{2}M$, we get $u \leq \frac{M}{R\sqrt{1-x^2}} \leq \frac{1}{\sqrt{2} \sqrt{1-\frac{1}{2}}}= 1$  for $ |x|\leq \frac{\sqrt{2}}{2}$. So it suffices to show that
\be
\begin{split}
P(u,x) >0,\qquad |x|\leq \frac{\sqrt{2}}{2},\quad 0<u\leq1. 
\end{split}
\ee
This is proved as follows. We have the identity 
\be\label{estimateP}
\begin{split}
P(u,x)= 1+5u^2 + 3u(1-u^2) x +12x^2 u^2 \left(x+\frac{u^2-6}{12 u}\right)^2-x^2\left(\frac{u^4}{12}+10 u^2\right). 
\end{split}
\ee
From (\ref{estimateP}), we have the inequality 
 \be\label{estimatePP}
\begin{split}
P(u,x)&\geq 1+5u^2  -\frac{3\sqrt{2}}{2} u (1-u^2) -\frac{1}{2}\left(\frac{u^4}{12}+10 u^2\right) \\
&= 1 -\frac{3\sqrt{2}}{2} u (1-u^2)  -\frac{1}{24}u^4 \\ 
&\geq \frac{23}{24}  -\frac{3\sqrt{2}}{2} u (1-u^2), \qquad|x|\leq \frac{\sqrt{2}}{2},\quad 0<u\leq1. 
\end{split}
\ee
It is easy by calculus to show that the right hand side is strictly positive.  Hence \eqref{3} follows.  

\eqref{2} By \eqref{1d} it suffices to show that 
\be\label{estimate1}
x+(1-2x^2) u >0, \qquad 0<u\leq\frac{M}{R \sqrt{1-x^2}}, ~\frac{\sqrt{2}}{2}\leq x \leq  1- \frac{M^2}{2 R^2}.  
\ee
To show this claim, first note that $x+(1-2x^2) u \geq x-\frac{M (2x^2-1)}{R \sqrt{1-x^2}}$. We then put $y=(\frac{M}{R})^2\in(0,\frac{1}{2}]$ and consider the function
\be
f(x,y):=(1-x^2)\left(x^2-\left(\frac{\sqrt{y} (2x^2-1)}{\sqrt{1-x^2}}\right)^2\right)=  -(1+4y)x^4+(1+4y)x^2-y. 
\ee
The function $x\mapsto f(x,y)$ is strictly decreasing on $[\frac{\sqrt{2}}{2}, 1- \frac{y}{2}]$, and so 
\be
f(x,y) \geq f(1-y/2, y) = \frac{y^2}{16}(44-72y+31y^2-4y^3). 
\ee
By calculus, the function on the right hand side is strictly positive on $(0,\frac{1}{2}]$. Hence $f(x,y)>0$ and thus we obtain (\ref{estimate1}). 
 

\vspace{10pt}
\noindent
{\bf Step 5}. 
Steps 2--4 imply that the equation $A_\nu(R\sin(\theta)e^{\ri \theta})=\frac{1}{s}$ has at most two solutions $\theta \in(0,\pi)$ for each fixed $R\in(0,\infty)$ and $s\geq T$. Hence $\mu^{\boxplus s}$ is unimodal by Lemma \ref{lemma-density}.  


In general, let $\mu$ be an FID measure with free characteristic triplet $(\eta, a, \nu)$ such that $\supp(\nu) \subset [-M,M]$. 
There exist functions $k_n\in C^\infty(-\infty,\infty)$ such that $\supp(k_n) \subset [-M-\frac{1}{n},M+\frac{1}{n}]$, $k_n>0$ in a neighborhood of $0$ and 
\be
\frac{|t| k_n(t)}{1+t^2} \overset{\rm w}{\longrightarrow} \frac{t^2}{1+t^2}\,\nu(\d t) + a \delta_0,\qquad n\to\infty. 
\ee
Let $\mu_n$ be the FID probability measure corresponding to $(\eta,0, \frac{k_n(t)}{\vert t\rvert}\6t)$. 
From \cite[Theorem 3.8]{B-NT02}, we have that 
\be
\mu_n^{\boxplus s} \overset{\rm w}{\longrightarrow} \mu^{\boxplus s},\qquad s>0.  
\ee
Note then that 
\be
\sigma^2(\mu_n)=\int_{-M-\frac{1}{n}}^{M+\frac{1}{n}} |t| k_n(t)\6t \overset{\rm w}{\longrightarrow}\int_{-M}^M t^2\,\nu(\d t) +a = \sigma^2(\mu).  
\ee
We know that $\mu_n^{\boxplus s}$ is unimodal for $s \geq \frac{4 (M+n^{-1})^2}{\sigma^2(\mu_n)}$. Since the unimodality is preserved by weak convergence, $\mu^{\boxplus s}$ is unimodal for $s\geq \frac{4 M^2}{\sigma^2(\mu)}$.

\vspace{10pt}
\noindent
{\bf Step 6} (Optimality of the constant 4). Let $C\leq 4$ be the optimal constant such that $\mu^{\boxplus s}$ is unimodal for 
\be
s \geq \frac{C \sup\{|x|^2 : x\in \supp(\nu)\}}{\sigma^2(\mu)},
\ee
 where $\nu$ is the free L\'evy measure of $\mu$.  
Let $\mu_b$ be the compound free Poisson distribution defined by 
\be
\mathcal{C}_{\mu_b}(z)= \frac{b z}{1-z}+\frac{z}{1+z},\qquad z\in\C^-, b>0.
\ee
The free L\'evy measure is given by 
\be
\nu_b=b \delta_1+\delta_{-1}. 
\ee
Let $\psi_{b,s}:=\psi_{\mu_b^{\boxplus s}}, V_{b,s}:=V_{\mu_b^{\boxplus s}}$ for simplicity. Due to Theorem \ref{thm Hu} of Huang, the support of the absolutely continuous part of $\mu_b^{\boxplus s}$ is $\psi_{b,s}(\overline{V_{b,s}})$, where 
\be
V_{b,s}=\left\{x\in\R {~\bigg|~} \frac{b}{(x-1)^2}+\frac{1}{(x+1)^2} >\frac{1}{s}\right\}, \qquad s>0, 
\ee 
and $\psi_{b,s}$ is a homeomorphism of $\R$. 
It is clear that $\{\pm 1\}\in V_{b,s}$, and moreover, if $\ep>0$ is small enough then we find that  $1-\ep\notin V_{\ep^4,4-5\ep}$. This implies that the support of $\mu_{\ep^4}^{\boxplus (4-5\ep)}$ has at least two connected components, and hence $\mu_{\ep^4}^{\boxplus (4-5\ep)}$ is not unimodal for small $\ep>0$.  Since $\sigma^2(\mu_b)=\int_{\{-1,1\}}t^2\,\nu_b(\d t)=1+b$ and $\supp(\nu_b) \subset [-1,1]$, we get $C\geq (4-5\ep) (1+\ep^4)$ and hence $C\geq4$ by letting $\ep \downarrow0$.  
\end{proof}


There exists an FID measure $\mu$ such that $\mu^{\boxplus s}$ is not unimodal for any $s>0$. We can take $\mu$ even to have finite moments of all orders. The construction is similar to \cite[Proposition 4.13]{Hu1}. 
\begin{example}\label{unbounded}
Let $\mu$ be the FID measure defined by 
\be
\mathcal{C}_\mu(z)= \sum_{n=1}^\infty \frac{a_n b_n z}{1- b_n z},   
\ee
where 
\begin{alignat}{2}
&a_n,b_n>0,  & &n \geq1,\\  
&b_{n+1}- b_{n}>0, & & n\geq 1, \\
&\lim_{n\to\infty}(b_{n+1} -b_{n})=\infty, \\
&\int_{0}^\infty t^2\,\nu(\d t)=\sum_{n=1}^\infty  b_n^2a_n<\infty. 
\end{alignat}
The free L\'evy measure is given by 
$
\sum_{n=1}^\infty a_n \delta_{b_n},  
$
so 
\be
A_\nu(x+\ri y)= \sum_{n\geq1}\frac{a_n b_n^2}{(x-b_n)^2+y^2},\qquad x\in\R, y\geq0. 
\ee
Let $v_s:= v_{\mu^{\boxplus s}}$ and  $\psi_s:=\psi_{\mu^{\boxplus s}}$ for simplicity. Recall that 
\be
v_s(x)=\inf\left\{y>0 ~\bigg|~ A_\nu(x+\ri y) <\frac{1}{s} \right\}. 
\ee
Let $x_k:= \frac{b_k+b_{k+1}}{2}$. 
Then $|x_k-b_n| \geq \frac{b_{k+1}-b_k}{2}$ for any $k, n \geq1$. Hence 
\be
\begin{split}
A_\nu(x_k) \leq \left(\frac{2}{b_{k+1}-b_k}\right)^2\sum_{n=1}^\infty b_n^2 a_n \to 0,\qquad k\to\infty. 
\end{split} 
\ee
This implies that for any $s>0$, there exists $K=K(s)\in\N$ such that $A_\nu(x_k)<\frac{1}{s}$ for all $k\geq K$. 
Hence 
\be
\psi_s(x_k)\notin \psi_s(\overline{\{x \in \R \mid v_s(x)>0 \}})={\rm supp}((\mu^{\boxplus s})^{\rm ac}) 
\ee
for $k \geq K$. Since $\mu^{\boxplus s}$ has at most one atom, $\psi_s(x_k)\notin {\rm supp}(\mu^{\boxplus s})$ for infinitely many $k$. Since $A_\nu(b_k)=\infty$, so $\psi_s(b_k) \in {\rm supp}(\mu^{\boxplus s})$ for any $s>0$ and any $k\in \N$. Therefore the support of $\mu^{\boxplus s}$ consists of  infinitely many connected components for any $s>0$. This in particular implies that $\mu^{\boxplus s}$ is not unimodal for any $s>0$. In the specific case $b_n=2^n$ and $a_n=2^{-n^2}$, the free cumulants are all finite: 
\be
\int_0^\infty t^{2m} \,\nu(\d t) = \sum_{n=1}^\infty 2^{-n^2 +2m n} <\infty,\qquad m \in\N. 
\ee
This implies that $\mu$ has finite moments of all orders \cite{BG06}. 
\end{example}

Since the partial free convolution semigroup $(\mu^{\boxplus s})_{s\geq1}$ can be defined for all probability measures $\mu$ on $\R$ (see \cite{NS96,BB04}), a similar question can be considered for non FID distributions. 
\begin{conjecture}
If $\mu$ is a compactly supported probability measure on $\R$, then there exists $T\geq 1$ depending on $\mu$ such that $\mu^{\boxplus s}$ is unimodal for $s \geq T$. 
\end{conjecture}

Example \ref{unbounded} constructs a probability measure $\mu$ with infinite connected components such that $\mu^{\boxplus s}$ is not unimodal for any $s>0$. When the number of connected components is finite, there is a possibility of extending Theorem \ref{unimodal-large} to measures with unbounded support. 
\begin{problem}  
 Let $\mu$ be an FID (or not) probability measure whose support has a finite number of connected components. 
Does there exist $T~(>1)$ such that $\mu^{\boxplus s}$ is unimodal for $s\geq T$?   
\end{problem}





\section*{Acknowledgements} T.\ Hasebe was supported by Marie Curie Actions -- International Incoming Fellowships 328112 ICNCP and by Grant-in-Aid for Young Scientists (B) 15K17549, JSPS. 
N. Sakuma is supported by Scientific Research(C) 15K04923, JSPS.

\vspace{1.5cm}

\begin{minipage}[c]{0.5\textwidth}
Takahiro Hasebe \\
Department of Mathematics \\
 Hokkaido University\\
 Kita 10, Nishi 8, Kita-ku \\
 Sapporo 060-0810 \\
 Japan\\
 {\tt thasebe@math.sci.hokudai.ac.jp}
\end{minipage}
\hfill
\begin{minipage}[c]{0.5\textwidth}
Noriyoshi Sakuma\\
Department of Mathematics \\
Aichi University of Education\\
1 Hirosawa, Igaya-cho\\
 Kariya-shi, 448-8542 \\
  Japan\\
{\tt sakuma@auecc.aichi-edu.ac.jp}
\end{minipage}

\end{document}